\documentclass[a4paper,12pt,reqno]{article}
\usepackage[T1]{fontenc}
\usepackage{authblk}
\usepackage{epsfig,graphicx,subfigure}
\usepackage{amsthm,amsmath,amssymb,amsfonts}
\usepackage{color,xcolor}
\usepackage[all]{xy}
\usepackage[top=40mm, bottom=40mm, left=40mm, right=35mm]{geometry}
\usepackage{tikz-cd}
\usepackage{hyperref}

\newtheorem{theorem}{Theorem}[section]
\newtheorem{lem}[theorem]{Lemma}
\newtheorem{proposition}[theorem]{Proposition}
\newtheorem{corollary}[theorem]{Corollary}
\theoremstyle{definition}

\numberwithin{equation}{section}

\begin{document}
\setcounter{page}{1}
\title{\vspace{-1.5cm}
\vspace{.5cm}
{\large{\bf    On scalable $K$-frames and a version of Lax-Milgram theorem} }  }
 \date{}
\author{{\small \vspace{-2mm}  F. Javadi$^1$ and M. J. Mehdipour$^{1} $\footnote{Corresponding author} }}
\affil{\small{\vspace{-4mm}  $^1$ Department of Mathematics, Shiraz University of Technology, P. O. Box $71555-313$, Shiraz,  Iran.}}
\affil{\small{\vspace{-4mm} f.javadi@sutech.ac.ir}}
\affil{\small{\vspace{-4mm} mehdipour@sutech.ac.ir}}
\maketitle
\hrule
\begin{abstract}
\noindent
In this paper, we first prove a theorem by a  little modification on  the Lax-Milgram theorem. Then, using $K$-frames, we obtain  lower and upper bounds for the results obtained from this theorem.
Also, we present some methods for the characterization of scalable $K$-frames.
 Finally, we introduce piecewise scalable $K$-frames and give necessary and sufficient conditions for a $K$-frame to be piecewise scalable.
 \end{abstract}

\\
\hrule
\vspace{0.5 cm}
\baselineskip=.55cm
\section{Introduction}
Let $H$  denote a separable infinite dimensional Hilbert space  and
${\cal L}(H)$ be  the space  of  all bounded linear operators from  $H$ into $ H$.
A sequence  $\{f_j\}_{j \in \Bbb{Z}}$   is called a \textit{frame}, if there exist constants $A, B > 0$ such that for every $f \in H$,
\begin{equation} \label{f1}
A \, \|  f \|^2 \leq \sum_{j \in \Bbb{Z}} |\langle  f, f_j \rangle|^2 \leq  B\, \|f \|^2;
\end{equation}
see [7] for more details. 
If $A = B = 1$,
 then we call $\{f_j\}_{j \in \Bbb{Z}}$   a  \textit{Parseval frame}.
A frame $\{f_j\}_{j \in \Bbb{Z}} $   is called a
\textit{scalable frame} if there exist constants $\{c_j\} \in \ell^{\infty}(J)$ such that $\{c_j f_j\}_{j \in \Bbb{Z}}$ is a Parseval frame. So for every $f \in H$, we have
\begin{eqnarray*}
f= \sum_{j\in \Bbb{Z}} \langle f, c_j f_j \rangle \,  c_j f_j \quad \text{and} \quad \sum_{j \in \Bbb{Z}} |\langle  f, c_j f_j \rangle|^2= \|f \|^2.
\end{eqnarray*}

In this paper, we always assume that  $K \in \mathcal{L}(H)$.  A sequence  $\{f_j\}_{j \in \Bbb{Z}}$   is called a \textit{$K$-frame},
 if there exist constants $A, B > 0$ such that for every $f \in H$,
\begin{equation} \label{f1}
A \, \| K^* f \|^2 \leq \sum_{j \in \Bbb{Z}} |\langle  f, f_j \rangle|^2 \leq  B\, \|f \|^2.
\end{equation}
If
\begin{equation*}
 \| K^* f \|^2= \sum_{j \in \Bbb{Z}} |\langle  f,f_j \rangle|^2,
\end{equation*}
then $\{f_j\}_{j \in \Bbb{Z}}$  is called  a  \textit{Parseval $K$-frame} on $H$.
The operator  $T: H  \rightarrow  \ell^2(J)$ defined by
\begin{equation*}
Tf= \{ \langle f, f_j \rangle \}_{j \in \Bbb{Z}},
\end{equation*}
is called \textit{analysis operator of} $\{f_j\}_{j \in \Bbb{Z}}$.
The operator  $S: H \rightarrow H$ defined by
\begin{equation*}
Sf =T^* Tf = \sum_{j \in \Bbb{Z}} \langle f, f_j \rangle f_j,
\end{equation*}
is called \textit{frame operator} of $\{f_j\}_{j \in \Bbb{Z}}$; see  [9, 10, 11, 12, 15].  

It's clear that if $K$ is the identity map on $H$, then every $K$-frame is an ordinary frame. Hence, $K$-frames  arise naturally as a generalization of the ordinary frames. Parseval
 $K$-frames are one of the most important types of $K$-frames. They are applied in signal communication. It is natural to ask how to make a Parseval
 $K$-frame from a given $K$-frame. Kutyniok
 et. al [13] solved this question for frames and introduced the conception of scalable frames. This notion has been investigated by some authors [2,4,5,6,14]. These results cause us to investigate the conception of scalable $K$-frames.

 Several authors studied scalable frames and piecewise scalable frames. For example, Cahill and Chen [2]  investigated  the following question: Under what conditions is a frame scalable? Casazza et. al [3] generalized the notion of scalable frames and introduced  
piecewise scalable frames. They also characterized them. In this paper, we continue these investigations.

The outline of the paper is as follows.
 In Section 2, we prove an analogue of Lax-Milgram theorem and give an application of it for $K$-frames. 
 In Section 3, we construct a scalable $K$-frame by a given scalable and a bounded operator. 
In Section 4, we introduce piecewise scalable $K$-frame and characterize them.                                                           
\section{Generalization of Lax-Milgram Theorem}
Let ${\cal H}$ be a real Hilbert space and
 $\sigma: {\cal H}\times {\cal H}\rightarrow \Bbb{R}$ be a bilinear map.
 For $f_0\in {\cal H}$ and  a non-empty closed convex subset $C$ of ${\cal H}$,
we define the function $J_{f_0, \sigma}: C\rightarrow {\Bbb R}$ by
\begin{eqnarray*}
J_{f_0,\sigma}(v)=\frac{1}{2} \sigma(v,v) - \langle KK^*f_0,  v \rangle.
\end{eqnarray*}
Let us recall that $\sigma$ is called \textit{coercive} if there exists a constant $\alpha >0$ such that  for every $v \in {\cal H}$
\begin{eqnarray*}
\sigma(v , v) \geq \alpha  \, \|v \|^2,
\end{eqnarray*}
and is called \textit{continuous} if there exists $\beta>0$ such that for every $u, v \in {\cal H}$
\begin{eqnarray*}
 |\sigma(u , v)| \leq \beta  \, \|u \| \,  \|v\|.
\end{eqnarray*}
Also,
 a self-adjoint operator $A: {\cal H}\rightarrow {\cal H}$ is called \textit{coercive} if there exists a constant $\alpha >0$ such that  for every $v \in {\cal H}$
\begin{eqnarray*}
\langle Av, v \rangle \geq \alpha  \, \|v\|^2.
\end{eqnarray*}

The following result is a little modification of the Stampacchia and the Lax-Milgram theorems [1].
\begin{theorem} \label{J21}
Let  $\sigma$ be a  continuous coercive bilinear map on ${\cal H}$, and $C$ be a non-empty closed convex subset of ${\cal H}$. Then for every bounded linear functional $L$ on ${\cal H}$,
there exist  unique $u_0 \in C$ and $f_0 \in {\cal H}$ such that for every $v\in {\cal H}$
\begin{eqnarray*} \label{f6}
\sigma(u_0, v-u_0) \geq \langle KK^* f_0, v-u_0 \rangle \quad \emph{and} \quad  L(v) = \langle f_0, v \rangle.
\end{eqnarray*}
Moreover, if $\sigma$ is symmetric,
 then $u_0$ is characterized by the property
 \begin{eqnarray*} \label{f7}
J_{f_0,\sigma}(u_0)= \min_{v \in C}J_{f_0,\sigma}(v).
\end{eqnarray*}
\end{theorem}
\begin{proof}
Let  $L$ be a bounded linear functional on ${\cal H}$.
 Using  the Riesz
representation theorem,
 there exists a unique element $f_0 \in  {\cal H}$ such that for every $v\in {\cal H}$,
\begin{eqnarray*}
L(v)= \langle f_0, v \rangle.
\end{eqnarray*}
Put $\tilde{L}:=L K K^*$.
Then there exists $f_0^{\prime}\in {\cal H}$ such that  for every $v\in {\cal H}$, we have
$\tilde{L}(v)= \langle f_0^{\prime}, v \rangle$.
So for every $v\in {\cal H}$, we obtain
\begin{eqnarray*}
\langle f_0^{\prime}, v \rangle &=& \tilde{L} (v) = L KK^*(v) \\
&=&L(KK^*(v)) \\
&=& \langle f_0, KK^* v \rangle \\
&=& \langle KK^*f_0, v \rangle.
\end{eqnarray*}
For every $u  \in  {\cal H}$,
the linear functional $v\mapsto\sigma(u,v)$ is continuous.
So,
 we find some unique element in ${\cal H}$, denoted by $\Lambda u$,
such that  $\sigma(u, v) = \langle \Lambda u, v \rangle$ for every $v \in {\cal H}$.
Then   the operator $\Lambda $ is  linear and there exists a constant $\beta>0$ such that for every $u \in {\cal H}$
\begin{eqnarray} \label{f22}
\|\Lambda u\| ^2&=& | \langle \Lambda u, u \rangle |\\
&=& |\sigma(u, u)| \leq \beta  \, \|u \|^2. \nonumber
\end{eqnarray}
Since $\sigma$ is coercive, there exists a constant $\alpha>0$ such that for every $u \in {\cal H}$
\begin{eqnarray} \label{f220}
\langle \Lambda u, u \rangle=\sigma(u,u)  \geq \alpha  \|u\|^2.
\end{eqnarray}
Set $\gamma:=\alpha/ \beta^2$ and define the map $T$ on $C$ by
\begin{eqnarray*}
T(v)= P_C( \gamma KK^* f_0 -\gamma \Lambda v+v),
\end{eqnarray*}
where  $P_C$ is the projection onto the closed convex set $C$.  Note that every projection on $C$ is not increase distance;
see Proposition 5.3 in [1].
Thus
\begin{eqnarray*}
\|Tv_1 - Tv_2\| &=& \| P_C( \gamma KK^* f_0 -\gamma \Lambda v_1+v_1) -  P_C( \gamma KK^* f_0 -\gamma \Lambda v_2+v_2) \| \\
&\leq& \| (\gamma KK^* f_0 -\gamma \Lambda v_1+v_1) -  (\gamma KK^* f_0 -\gamma \Lambda v_2+v_2) \|\\
&\leq & \|(v_1 -v_2) - \gamma (\Lambda v_1 - \Lambda v_2)\|.
\end{eqnarray*}
So, using \eqref{f22} and \eqref{f220}, we obtain
\begin{eqnarray*}
\|Tv_1 - Tv_2\|^2 &=&  \|(v_1 -v_2)\|^2 - 2  \gamma \langle \Lambda v_1 - \Lambda v_2, v_1-v_2\rangle+ \gamma^2 \|\Lambda v_1 - \Lambda v_2\|^2 \\
&\leq&   \|(v_1 -v_2)\|^2 (1-2 \gamma \alpha + \gamma^2 \beta^2).
\end{eqnarray*}
Hence  $T$  is a strict contraction.
 In view of Theorem 5.7 in [1],
 $T$ has a unique fixed point,
i.e.,
there exists a unique $u_0 \in C$ such that
\begin{eqnarray*}
u_0=T(u_0)= P_C( \gamma KK^* f_0 -\gamma \Lambda u_0+u_0).
\end{eqnarray*}
This implies that
\begin{eqnarray*}  \label{f11}
\langle  \gamma KK^*f_0  - \gamma \Lambda u_0 +u_0 -u_0, v-u_0 \rangle \leq 0;
\end{eqnarray*}
or equivalently,
\begin{eqnarray*}  \label{f10}
\langle \Lambda u_0, v-u_0 \rangle \geq \langle KK^*f_0, v-u_0 \rangle.
\end{eqnarray*}
Therefore, $\sigma(u_0, v-u_0) \geq \langle KK^* f_0, v-u_0 \rangle$.

Now, let  $\sigma(u, v)$ be  symmetric.
Then the mapping $$(u, v)\mapsto \sigma(u, v)=\langle \Lambda u, v \rangle$$ defines an inner product on ${\cal H}$.
With this inner product,  ${\cal H}$  is also a Hilbert space.
By
the Riesz theorem, there exists a unique element  $g_0  \in  {\cal H}$  such that for every $v \in {\cal H}$
\begin{eqnarray*}
\tilde{L}(v)=\langle KK^*f_0 , v \rangle = \sigma(g_0, v).
\end{eqnarray*}
The rest of the proof is similar to the proof of Theorem Stampacchia; see Theorem 5.6 in [1].
So we omit it.
\end{proof}
Let $\sigma(u, v)$ be a continuous bilinear  on ${\cal H}$.
 Then for every $v \in {\cal H}$, we have
\begin{eqnarray*}
\sigma(v,v) \geq 0.
\end{eqnarray*}
Thus the function  $v \mapsto \sigma(v,v)$  is convex.
 This fact together with Theorem \ref{J21} proves the following result.
\begin{corollary} \label{j7}
Let $\sigma(u,v)$ be a continuous coercive bilinear on ${\cal H}$.
 Then for every bounded linear functional $L$ on ${\cal H}$,
 there exist unique $u_0, f_0 \in {\cal H}$  such that for every $v\in {\cal H}$
\begin{eqnarray*} \label{f6}
\sigma(u_0, v-u_0)= \langle KK^* f_0, v-u_0 \rangle \quad \emph{and} \quad  L(v) = \langle f_0, v \rangle.
\end{eqnarray*}
Moreover, if $\sigma$ is symmetric,
then $u_0$ is characterized by the property
 \begin{eqnarray*} \label{f7}
J_{f_0,\sigma}(u_0)= \min_{v \in C}J_{f_0,\sigma}(v).
\end{eqnarray*}
\end{corollary}
Let  $\{f_j\}_{j \in \Bbb{Z}}$  be a  $K$-frame for ${\cal H}$ with the frame operator   $S$.
Define the bilinear form $\sigma_S: {\cal H} \times {\cal H}  \rightarrow  \Bbb{C}$ by
\begin{eqnarray} \label{f16}
\sigma_S(u,v) = \langle Su, v \rangle.
\end{eqnarray}
Then $\sigma_S$  is bounded, coercive and symmetric.
So  Corollary \ref{j7} yields the following:
\begin{corollary}
Let  $\{f_j\}_{j \in \Bbb{Z}}$  be a  $K$-frame for ${\cal H}$ with the frame operator $S$.
Then for every bounded linear functional $L$ on ${\cal H}$, there exist unique  $u_0, f_0\in {\cal H}$ for which
\begin{eqnarray*}
\langle Su_0,v\rangle = \langle KK^*f_0, v \rangle \quad \emph{and} \quad  L(v) = \langle f_0, v \rangle
\end{eqnarray*}
 for all  $v \in  H$. Furthermore,
 \begin{eqnarray*} \label{f20}
J_{f_0,\sigma_S}(u_0)= \min_{v \in C}J_{f_0,\sigma_S}(v).
\end{eqnarray*}
\end{corollary}
In the following, using  $K$-frames, we get  lower  and upper  bounds for $ \min_{v \in {\cal H}}  J_{f_0,\sigma_S}$. 
\begin{theorem}
Let  $\{f_j\}_{j \in \Bbb{Z}}$  be a  $K$-frame for ${\cal H}$ with the frame operator $S$ and $K$-frame constants  $A < B$. If $f_0\in {\cal H}$, then
\begin{eqnarray*}
- \frac{\|K^*f_0 \|^2}{2A} \leq \min_{v \in {\cal H}}  J_{f_0,\sigma_S}(v) \leq  -\frac{7}{32}\frac{ \| K^*f_0\|^2}{B\|f_0\|^2}\leq -\beta\|K^*f_0\|^2,
\end{eqnarray*}
for some $\beta>0$.
\end{theorem}
\begin{proof}
Let $\{f_j\}_{j \in \Bbb{Z}}$ be a $K$-frame for $\cal{H}$.
 By  \eqref{f1},   for every  $v  \in \cal{H}$,
we have
\begin{eqnarray*}
J_{f_0,\sigma_S}(v)  = \bigg( \frac{1}{2} \langle Sv, v \rangle - \langle KK^*f_0, v \rangle \bigg)
& \geq& \bigg( \frac{A}{2} \|K^*v\|^2- \langle KK^*f_0 , v \rangle \bigg).
\end{eqnarray*}
Apply Cauchy-Schwartz inequality to conclude that
\begin{eqnarray*}
\langle KK^*f_0, v \rangle  =\langle K^*f_0, K^*v \rangle  \leq \|K^*f_0\| \|K^*v\|.
\end{eqnarray*}
So we have
\begin{eqnarray*}
J_{f_0,\sigma_S}(v)  \geq  \bigg( \frac{A}{2} \|K^*v\|^2- \|K^*f_0\| \|K^*v\| \bigg).
\end{eqnarray*}
The minimum of the function $v\mapsto \frac{A}{2} \|K^*v\|^2- \|K^*f_0\| \|K^*v\|$ is attained in point $ \frac{\|K^*f_0 \|}{A}$, and hence
\begin{eqnarray*}
\min_{v \in {\cal H}} J_{f_0,\sigma_S}(v)  &\geq&
 \bigg( \frac{A}{2} \frac{1}{A^2} \|K^*f_0\|^2- \|K^*f_0\| \frac{\|K^*f_0\|}{A} \bigg) \\
&=&- \frac{\|K^* f_0\|^2}{2A}.
\end{eqnarray*}
In view of  \eqref{f1},  for every  $v  \in  \cal{H} $ we have
\begin{eqnarray*}
J_{f_0,\sigma_S}(v)  = \bigg( \frac{1}{2} \langle Sv, v \rangle - \langle KK^*f_0, v \rangle \bigg)
& \leq & \bigg( \frac{B}{2} \| v\|^2- \langle KK^*f_0 , v \rangle \bigg).
\end{eqnarray*}
Thus
\begin{eqnarray*}
\min_{v \in {\cal H}} J_{f_0,\sigma_S}(v) &\leq& \min_{v \in {\cal H}}  \bigg( \frac{B}{2} \| v\|^2- \langle KK^*f_0, v \rangle \bigg)\\
&\leq& \frac{B \alpha^2}{2} \|f_0\|^2 - \alpha\|K^*f_0\|^2\\
&<&0,
\end{eqnarray*}
whenever
\begin{eqnarray*}
0<\alpha < \frac{2}{B} \frac{\|K^*f_0\|^2}{\|f_0\|^2}.
\end{eqnarray*}
The minimum of the function $\alpha\mapsto \frac{B \alpha^2}{2} \|f_0\|^2 - \alpha\|K^*f_0\|^2$ is attained in point $\alpha=\frac{\|K^*f_0\|^2}{4 B \|f_0\|^2}$. Therefore,
\begin{eqnarray*}
\frac{B}{2} \, \alpha^2 \| f_0 \|^2- \alpha \|K^*f_0\|^2
 &=&  \frac{B}{2} \frac{\|K^*f_0\|^4}{16 B^2 \|f_0\|^4} \|f_0\|^2 - \frac{\|K^*f_0\|^2}{4 B \|f_0\|^2} \|K^*f_0\|^2 \\
&=& \frac{1}{2} \frac{\|K^*f_0\|^2}{16 B \|f_0\|^2}- \frac{\|K^*f_0\|^4}{4B \|f_0\|^2} \\
&=& - \frac{7}{32} \frac{\|K^*f_0\|^4}{B \|f_0\|^2},
\end{eqnarray*}
as claimed.
\end{proof}

 \section{Construction of Scalable $K$-Frames}
Throughout this section, $H$  denotes a separable infinite dimensional Hilbert space. First, we  give a slight generalization of the concept of scalable frames for $K$-frames;
in fact, a $K$-frame $\{f_j\}_{j \in \Bbb{Z}} \subseteq H$   is called a \textit{$K_s$-frame with scalling $\{c_j\}\in \ell^{\infty}(J)$},
briefly $K_s$-frame and denoted by $(f_j, c_j)_{j \in \Bbb{Z}}$,
 if  $\{c_j f_j\}_{j \in \Bbb{Z}}$ is a \emph{Parseval $K$-frame}, i.e.,
for every $f \in H$, we have
\begin{eqnarray*}
\sum_{j \in \Bbb{Z}} \langle f, c_j f_j \rangle \,  c_j f_j = KK^*f;
\end{eqnarray*}
or equivalently,
\begin{equation*}
 \sum_{j \in \Bbb{Z}} |\langle  f,c_jf_j \rangle|^2= \| K^* f \|^2.
\end{equation*}
\begin{theorem}
Let $U, V \in \mathcal{L}(H)$  and $\{f_j\}_{j \in \Bbb{Z}}$ be a sequence in $H$.
Then the following statements hold.

\emph{(i)}
If $\{f_j\}_{j \in \Bbb{Z}}$ is a scalable frame,
then $\{U f_j\}_{j \in \Bbb{Z}}$ is a  $U_s$-frame with the same scalling of  $\{f_j\}_{j \in \Bbb{Z}}$.

\emph{(ii)}
If $\{ f_j\}_{j \in \Bbb{Z}}$ is a $U_s$-frame,
 then  sequence  $\{V f_j\}_{j \in \Bbb{Z}}$ is a  $VU_s$-frame with the same scalling of  $\{f_j\}_{j \in \Bbb{Z}}$.

\emph{(iii)} If $\{f_j \}_{j \in \Bbb{Z}}$ and  $\{U f_j \}_{j \in \Bbb{Z}}$  are $K_s$-frames with the same scallings,
then
$\{f_j \}_{j \in \Bbb{Z}}$ is a $UK_s$-frame with the same scalling of  $\{f_j\}_{j \in \Bbb{Z}}$.

\emph{(iv)} If $\{f_j \}_{j \in \Bbb{Z}}$ is a $K_s$-frame and $N \geq 1$,
then
$\{K^N f_j \}_{j \in \Bbb{Z}}$ is a $K_s^{N+1}$-frame with the same scalling of  $\{f_j\}_{j \in \Bbb{Z}}$.
\end{theorem}
\begin{proof}
Let $T \in \mathcal{L}(H)$ and $\{c_j\}$ be a sequence in $\ell^{\infty}(J)$. Then for every $f\in H$, we have
\begin{eqnarray*}
 \sum_{j \in \Bbb{Z}}  \langle  f, c_j Tf_j \rangle \,  c_j Tf_j &=& T \bigg(  \sum_{j \in \Bbb{Z}}  \langle  f, c_j Tf_j \rangle \,  c_j f_j \bigg) \\
&=& T \bigg(  \sum_{j \in \Bbb{Z}}  \langle  T^* f, c_j f_j \rangle \,  c_j f_j \bigg).
\end{eqnarray*}
This proves (i) and (ii) if we set $T=U$ and $T=V$, respectively.

Let $(f_j, c_j)_{j \in \Bbb{Z}}$ and  $(Uf_j, c_j)_{j \in \Bbb{Z}}$ be $K_s$-frames. Then
\begin{eqnarray*}
(UK) (UK)^* f &=& \sum_{j \in \Bbb{Z}} \langle  f, c_j U f_j \rangle \,  c_j U f_j\\
&=& KK^*f\\
&=&\sum_{j \in \Bbb{Z}} \langle  f, c_j f_j \rangle \,  c_j f_j.
\end{eqnarray*}
So, $(f_j, c_j)_{j \in \Bbb{Z}}$ is  $UK_s$-frame.
Thus (iii) holds.

Let $(f_j, c_j)_{j \in \Bbb{Z}}$ be a $K_s$-frame.
Then for every $f \in H$, we have
\begin{eqnarray*}
\sum_{j \in \Bbb{Z}} |\langle  f, c_j K^N f_j \rangle|^2 &=&
 \sum_{j \in \Bbb{Z}} |\langle  (K^N)^* f, c_j  f_j \rangle|^2 \\
&=& \| K^* ({K^{N}})^* f \|^2 \\
&=&  \|( K^{N+1})^* f) \|^2.
\end{eqnarray*}
Therefore, $(K^Nf_j, c_j)_{j \in \Bbb{Z}}$ is $K^{N+1}_s$-frame.
That is, (iv) is true.
\end{proof}
\begin{theorem}
Let $T\in \mathcal{L}(H)$ and  $(f_j, c_j)_{j \in \Bbb{Z}}$ be a $K_s$-frame.
If $TK=KT$ and $T^*$ is isometry,
 then $(Tf_j, c_j)_{j \in \Bbb{Z}}$ is a $K_s$-frame.
\end{theorem}
\begin{proof}
Let  $(f_j, c_j)_{j \in \Bbb{Z}}$ be a $K_s$-frame.
Then
for every $f\in H$, we have
\begin{eqnarray*}
 \sum_{j \in \Bbb{Z}} |\langle  f, c_j Tf_j \rangle|^2 &=&
\sum_{j \in \Bbb{Z}} |\langle  T^* f, c_j f_j \rangle|^2  \\
&=& \|K^* T^*f \|^2 \\
&=&  \|(TK)^*f \|^2 \\
&=&  \|(KT)^*f \|^2 \\
&=&  \|T^* (K^*f ) \|^2 \\
&=&  \|K^*f \|^2.
\end{eqnarray*}
Thus $(Tf_j, c_j)_{j \in \Bbb{Z}}$ is a $K_s$-frame.
\end{proof}
\begin{theorem}
Let  $T\in \mathcal{L}(H)$ be an invertible operator,
$\{f_j \}_{j \in \Bbb{Z}}$ be a $K$-frame for $H$ and $\{c_j \}\in \ell^\infty(J)$.
Then  $(T f_j, c_j )_{j \in \Bbb{Z}}$ is a $K_s$-frame
 if and only if
the frame operator for $\{c_j Tf_j \}_{j \in \Bbb{Z}}$ is $(T^{-1}K) (T^{-1} K)^*$.
\end{theorem}
\begin{proof}
Let $S$ be the operator frame   for $\{c_j f_j \}_{j \in \Bbb{Z}}$.
Then for every $f \in H$, we have
\begin{eqnarray} \label{f2}
 \sum_{j \in \Bbb{Z}} \langle  f, c_j Tf_j \rangle \,  c_j Tf_j &=& T \bigg( \sum_{j \in \Bbb{Z}} \langle  T^* f, c_j f_j \rangle \,  c_j f_j \bigg) \nonumber  \\
&=& T S T^* f.
\end{eqnarray}
On the other hand, $(T f_j, c_j)_{j \in \Bbb{Z}}$ is a $K_s$-frame.
Thus
\begin{eqnarray} \label{f3}
 \sum_{j \in \Bbb{Z}} \langle  f, c_j Tf_j \rangle \,  c_j Tf_j = KK^*f.
\end{eqnarray}
By using \eqref{f2} and \eqref{f3},
$(T f_j, c_j)_{j \in \Bbb{Z}}$ is a $K_s$-frame
 if and only if $S= (T^{-1}K) (T^{-1} K)^*$.
\end{proof}
 \section{Piecewise Scalable $K$-Frames}

A $K$-frame $\{f_j\}_{j=1}^n$ for the  $n$-dimensional Hilbert space $\Bbb{R}^n$   is called a \textit{piecewise scalable $K$-frame}
(briefly, $K_{s}^p$-frame)
 if there exist an orthogonal projection $P: \Bbb{R}^n \rightarrow \Bbb{R}^n$ and  constants $\{a_j, b_j\}_{j=1}^n$ such that
 for every
$f \in \Bbb{R}^n$
\begin{eqnarray*}
\sum_{j =1}^n | \langle f, a_j P f_j + b_j (I-P) f_j \rangle |^2 = \|K^*f \|^2.
\end{eqnarray*}
 We denote this notion by $(f_j, P, a_j, b_j)_{j=1}^n$. In the following, let $P$ be an orthogonal projection on ${\Bbb R}^n$ and we set
\begin{eqnarray*}
X:= P(\Bbb{R}^n)\quad\hbox{and}\quad Y:= (I-P)(\Bbb{R}^n).
\end{eqnarray*}
One  can prove that  if  $(f_j, P, a_j, b_j)_{j=1}^n$ is a $K^p_s$-frame for $\Bbb{R}^n$,
 then $(P f _j, a_j)_{j=1}^n$ and $((I-P) f _j, b_j)_{j=1}^n$ are $K_s$-frames  for $X$ and $Y$, respectively. To prove other our results, we need the following lemma.
\begin{lem}\label{kl} Let  $\{f_j \}_{j=1}^n$  be a $K$-frame for $\Bbb{R}^n$ and  $KP=PK$.
  If $\{f_j \}_{j=1}^n$ is a  Parseval $K$-frame for $\Bbb{R}^n$, then $\{Pf_j \}_{j=1}^n$ is a  Parseval $K$-frame for $X$.
\end{lem}
\begin{proof}
For every $f \in X$, we have
\begin{eqnarray*}
\sum_{j=1}^n | \langle f , P f_j \rangle|^2 &=& \sum_{j=1}^n | \langle Pf ,  f_j \rangle|^2  \\
&=& \| K^* P f \|^2 \\
&=& \| PK^* f \|^2 \\
&=& \|K^* f\|^2.
\end{eqnarray*}
Also, note that if $PK=KP$,
then $\hbox{ran} (K^*) \subseteq \hbox{ran} (P)$.
\end{proof}
\begin{theorem} \label{J1}
	 Let  $\{f_j \}_{j=1}^n$  be a $K$-frame for $\Bbb{R}^n$ and  $KP=PK$.
Then $(f_j, P, a_j, b_j)_{j=1}^n$ is a $K^p_s$-frame
  if and only if

\emph{(i)}
		Both $(P f _j, a_j)_{j=1}^n$ and  $((I-P) f _j, b_j)_{j=1}^n$ are $K_s$-frame for $X$ and $Y$, respectively.

\emph{(ii)}
		For every $f \in \Bbb{R}^n$, we have
		\begin{eqnarray*}
			\sum_{j=1}^n a_j b_j \langle f, Pf_j \rangle \langle f, (I- P) f_j \rangle =0.
		\end{eqnarray*}
\end{theorem}
\begin{proof}
First, note that if (i) holds, then for every $f\in H$, we have
\begin{eqnarray}\label{mn}
	\sum_{j=1}^n | \langle f, a_j P f_j + b_j (I-P) f_j \rangle |^2  &=&   \sum_{j=1}^n | \langle f, a_j P f_j  \rangle|^2 +
	\sum_{j=1}^n  | \langle f, b_j (I-P) f_j \rangle |^2 \nonumber \\
	&+&  2 \sum_{j=1}^n a_j b_j \langle f, Pf_j \rangle \langle f, (I- P) f_j \rangle\nonumber \\
	&=& \| PK^*f \| + \| (I-P) K^*f \|^2 \\
	&+&  2 \sum_{j=1}^n a_j b_j \langle f, Pf_j \rangle \langle f, (I- P) f_j \rangle\nonumber\\
	&=&\|K^* f\|^2 +  2 \sum_{j=1}^n a_j b_j \langle f, Pf_j \rangle \langle f, (I- P) f_j \rangle\nonumber.
\end{eqnarray}

Assume now that $(f_j, P, a_j, b_j)_{j=1}^n$ is a $K^p_s$-frame.
	From Lemma \ref{kl} and the fact 
	\begin{eqnarray*}
		P(a_j P f_j + b_j (I-P) f_j)= a_j P f_j,
	\end{eqnarray*}
	we infer that $\{a_j Pf_j \}_{j=1}^n$ is a Parseval $K$-frame for $X$.
Therefore,  $\{ b_j (I-P) f _j\}_{j=1}^n$ is a Parseval $K$-frame for $Y$.
That is, (i) is proved.
 The converse follows at once from \eqref{mn}.
\end{proof}
\begin{corollary}
Let  $PK=KP$ and two of the following three statements be true. Then the third statement is also true.

\emph{(i)} $(f_j, P, a_j, b_j)_{j=1}^n$ is a $K^p_s$-frame.

\emph{(ii)}
Both $(P f _j, a_j)_{j=1}^n$ and  $((I-P) f _j, b_j)_{j=1}^n$ are $K_s$-frame for $X$ and $Y$, respectively.

\emph{(iii)}
For every $f \in \Bbb{R}^n$, we have
\begin{eqnarray*}
\sum_{j=1}^n a_j b_j \langle f, Pf_j \rangle \langle f, (I- P) f_j \rangle =0.
\end{eqnarray*}
\end{corollary}

In the next result, we let $T_1$ and $T_2$ be analysis operators of $\{Pf_j \}_{j=1}^n$ and $\{(I-P)f_j \}_{j=1}^n$, respectively.
\begin{corollary}\label{ea}
Let $\{f_j \}_{j=1}^n$  be a $K$-frame for $\Bbb{R}^n$ and $PK=KP$.
If there exists a proper non-empty subset $I$ of $\{1, ..., n\}$ such that $(Pf_j , c_j)_{j \in I} $ and $((I- P) f_j , d_j)_{ j \in I^{c}}
$ are $K_s$-frames for
 $X$ and $Y$, respectively, then there exist  constants $\{a_j\}_{j=1}^n$
and $\{b_j\}_{j=1}^n$
such that $(f_j, P, a_j, b_j)_{j=1}^n$  is a $K^p_s$-frame for $\Bbb{R}^n$.
\end{corollary}
\begin{proof}
Let $I$ be a proper non-empty subset of $\{1, ..., n\}$ such that $(Pf_j , c_j)_{j \in I} $ and $((I- P) f_j , d_j)_{ j \in I^{c}}
$ are $K_s$-frames for
 $X$ and $Y$, respectively.
Set
\begin{eqnarray*}
a_j=
\left\{
  \begin{array}{ll}
    c_j, & \hbox{$j \in I$;} \\
    0, & \hbox{$j \in I^c$}
  \end{array}
\right.\quad\hbox{and}   \quad\quad
b_j=
\left\{
  \begin{array}{ll}
    0, & \hbox{$j \in I$;} \\
    d_j, & \hbox{$j \in I^c$.}
  \end{array}
\right.
\end{eqnarray*}
Then for every $f \in X$, we have
\begin{eqnarray*}
\sum_{i =1}^n | \langle f, a_j Pf_j \rangle |^2 &=& \sum_{i \in I} | \langle f, a_j Pf_j \rangle |^2  + \sum_{i \in I^c} |  \langle f, a_j Pf_j \rangle |^2  \\
&=& \|K^* f\|^2.
\end{eqnarray*}
So $(Pf_j , a_j)_{j \in I} $ is a $K_s$-frame for
 $X$.
Similarly,
one can prove that  $((I- P) f_j , b_j)_{ j \in I^{c}}
$ is a $K_s$-frame for
 $Y$. But, $a_j b_j =0$ for all $j=1, ..., n$.
Then by Theorem \ref{J1}, $(f_j, P, a_j, b_j)_{j=1}^n$  is a $K^p_s$-frame for $\Bbb{R}^n$.
\end{proof}

As an immediate consequence of Theorem \ref{J1} and Corollary \ref{ea}, we have the following result.

\begin{corollary}
Let $\{f_j \}_{j=1}^n$  be a $K$-frame for $\Bbb{R}^n$ and $PK=KP$.
If there exists a non-empty proper subset $I$ of $\{1, ..., n\}$ such that $(Pf_j , c_j)_{j \in I} $ and $((I- P) f_j , d_j)_{ j \in I^{c}}
$ are $K_s$-frames for
 $X$ and $Y$, respectively, then $(Pf_j , c_j)_{j=1}^n $ and $((I- P) f_j , d_j)_{ j=1}^ n
$ are $K_s$-frames for
 $X$ and $Y$, respectively.
\end{corollary}

\begin{proposition}
Let $\{f_j \}_{j=1}^n$  be a $K$-frame for $\Bbb{R}^n$ and $PK=KP$.
Then the following assertions are equivalent.

\emph{(a)} $(f_j, P, a_j, b_j)_{j=1}^n$ is a $K^p_s$-frame.

\emph{(b)} If $D_1$ and $D_2$ are diagonal  operators on $\ell^2(n)$ with diagonal elements $\{a_j\}_{j=1}^n$  and $\{b_j\}_{j=1}^n$, respectively, then
\begin{eqnarray*}
T_1^* D_1^2 T_1= KK^*_{\mkern 1mu \vrule height 2ex\mkern2mu X}, ~~ T_2^* D_2^2 T_2= KK^*_{\mkern 1mu \vrule height 2ex\mkern2mu Y} ~~\emph{and}~~  T_1^* D_1 D_2 T_2=0.
\end{eqnarray*}
\end{proposition}
\begin{proof}
Let  $\{f_j \}_{j=1}^n$  be a  $K^p_s$-frame for $\Bbb{R}^n$.
By Theorem \ref{J1} (i), $\{a_j Pf_j \}_{j=1}^n$
 and $\{b_j(I-P)f_j \}_{j=1}^n$ are $K_s$-frame for $X$ and $Y$, respectively.
It follows that
$$
T_1^* D_1^2 T_1=S_1=KK^*_{\mkern 1mu \vrule height 2ex\mkern2mu X}\quad\quad\hbox{and}\quad\quad T_2^* D_2^2 T_2=S_2=KK^*_{\mkern 1mu \vrule height 2ex\mkern2mu Y},
$$
where $S_1$ and $S_2$ are the frame operators of $\{a_j Pf_j \}_{j=1}^n$ and $\{b_j(I-P)f_j \}_{j=1}^n$, respectively.
 On the other hand,
\begin{eqnarray}\label{l}
\langle T_1^* D_1 D_2 T_2(f), f\rangle=\bigg\langle \sum_{j=1}^n a_j b_j \langle f, Pf_j \rangle \langle f, (I- P) f_j\rangle, f\bigg\rangle
\end{eqnarray}
for every $f\in \Bbb{R}^n$.
Another application of Theorem \ref{J1}(ii),
shows that
\begin{eqnarray*}
\langle T_1^* D_1 D_2 T_2(f), f\rangle=0
\end{eqnarray*}
for every $f\in \Bbb{R}^n$. So (a)$\Rightarrow$(b).

To complete the proof, note that if $T_1^* D_1^2 T_1= KK^*_{\mkern 1mu \vrule height 2ex\mkern2mu X}$ and $T_2^* D_2^2 T_2= KK^*_{\mkern 1mu \vrule height 2ex\mkern2mu Y}$,
then $\{a_j Pf_j \}_{j=1}^n$ and $\{b_j(I-P)f_j \}_{j=1}^n$ are $K_s$-frame for $X$ and $Y$, respectively.
Also, if $T_1^* D_1 D_2 T_2=0$, then  by \eqref{l},
 we have
\begin{eqnarray*}
\sum_{j=1}^n a_j b_j \langle f, Pf_j \rangle \langle f, (I- P) f_j\rangle=0.
\end{eqnarray*}
Hence  (b)$\Rightarrow$(a).
\end{proof}
\begin{theorem}
Let $(f_j, P, a_j, b_j)_{j=1}^n$ be a $K^p_s$-frame.
Let $Q: \Bbb{R}^n \rightarrow \Bbb{R}^n$ be an orthogonal projection and $U: \Bbb{R}^n \rightarrow \Bbb{R}^n$ be a unitary operator such that
\begin{eqnarray*}
UP=QU \quad   \emph{and} \quad UK=KU.
\end{eqnarray*}
Then $(Uf_j, Q, a_j, b_j)_{j=1}^n$ is a $K^p_s-$frame.
\end{theorem}
\begin{proof}
Let $(f_j, P, a_j, b_j)_{j=1}^n$ be a $K^p_s$-frame. Then for every $f \in \Bbb{R}^n$, we obtain
\begin{eqnarray*}
\sum_{j=1}^n | \langle f, a_j Q \, Uf_j + b_j (I-Q) \, Uf_j \rangle |^2&=& \sum_{j=1}^n | \langle f, a_j U\,P f_j + b_j U \, (I-P) f_j \rangle |^2 \\
&=& \sum_{j=1}^n | \langle U^* f, a_j P f_j + b_j (I-P) f_j \rangle |^2 \\
&=& \|K^* \, U^* f \|^2 \\
&=& \| U^* \, K^* f\| ^2\\
&=& \langle U \, U^* K^*f, K^*f \rangle \\
&=& \|K^*f\|^2.
\end{eqnarray*}
So,  $(Uf_j, Q, a_j, b_j)_{j=1}^n$ is a $K^p_s-$frame.
\end{proof}
\section{Compliance with ethical standards}
\textbf{Conflict of interest:} All authors declare that they have no conflict of
interest.\\
\textbf{Data Availability Statement:} No data sets were generated or analyzed during the current study.

\end{document}